\documentclass[12pt]{article}

\usepackage{t1enc}
\usepackage[latin1]{inputenc}
\usepackage[english]{babel}
\usepackage{amsmath,amsthm}
\usepackage{amsfonts}
\usepackage{latexsym}
\usepackage{graphicx}
\usepackage{txfonts}
\usepackage[natural]{xcolor}

\newtheorem{thm}{Theorem}
\newtheorem{lem}[thm]{Lemma}

\newtheorem{conj}[thm]{Conjecture}

\makeatletter

\newcommand{\Rmnum}[1]{\expandafter\@slowromancap\romannumeral #1@}
\makeatother

\begin{document}

\title{Equitable vertex arboricity of planar graphs\thanks{Supported by the National Natural Science Foundation of China (No.\,11301410, 11201440, 11101243), the Natural Science Basic Research Plan in Shaanxi Province of China (No.\,2013JQ1002), the Specialized
Research Fund for the Doctoral Program of Higher Education (No.\,20130203120021), and the Fundamental Research Funds for the Central Universities (No.\,K5051370003, K5051370021).}}
\author{Xin Zhang\thanks{Email address: xzhang@xidian.edu.cn.}\\
{\small Department of Mathematics, Xidian University, Xi'an 710071, China}}
\date{}

\maketitle

\begin{abstract}\baselineskip  0.65cm
Let $G_1$ be a planar graph such that all cycles of length at most 4 are independent and let $G_2$ be a planar graph without 3-cycles and adjacent 4-cycles. It is proved that the set of vertices of $G_1$ and $G_2$ can be equitably partitioned into $t$ subsets for every $t\geq 3$ so that each subset induces a forest. These results partially confirm a conjecture of Wu, Zhang and Li \cite{WZL}.\\[.2em]

\end{abstract}

\baselineskip  0.65cm

\section{Introduction}

All graphs considered in this paper are finite, simple and undirected. By $V(G)$, $E(G)$, $\delta(G)$ and $\Delta(G)$, we denote the set of vertices, the set of edges, the minimum degree and the maximum degree of a graph $G$, respectively. For a plane graph $G$, $F(G)$ denotes its set of faces. A $k$-, $k^+$- and $k^-$-$vertex$ (resp.\,\emph{face}) in $G$ is a vertex (resp.\,face) of degree $k$, at least $k$ and at most $k$, respectively.
By $N(u)$, we denote the set of neighbors of $v$.
We call $u$ the \emph{$k$-neighbor} or \emph{$k^+$-neighbor} of $v$ if $uv\in E(G)$ and $u$ is a $k$-vertex or a $k^+$-vertex, respectively.
Two cycles are \emph{independent} in $G$ if they share no common vertices in $G$.
For other undefined notations, we refer the readers to \cite{Bondy.2008}.

The \emph{vertex arboricity}, or \emph{point arboricity} $a(G)$ of a graph $G$ is the minimum number of subsets into which the set of vertices can be partitioned so that each subset induces a forest. This chromatic parameter of graphs was extensively studied since it was
first introduced by Chartrand and Kronk in \cite{CK}, where is proved that $a(G)\leq 3$
for every planar graph. 

As we know, there are many variations of vertex arboricity of graphs, such as linear vertex arboricity \cite{M}, fractional vertex arboricity \cite{YZ}, fractional linear vertex arboricity \cite{ZWL} and tree arboricity \cite{CCC}.
Naturally, we can also consider the equitable version of
vertex arboricity when we restrict the partition in its original definition to be an equitable one, that is, a partition so that the size of each subset is either $\lceil|G|/k\rceil$ or $\lfloor|G|/k\rfloor$. If the set of vertices of a graph $G$ can be equitably partitioned into $k$ subsets such that each subset of vertices induce a forest of $G$, then we call that $G$ admits an \emph{equitable $k$-tree-coloring}. The minimum integer $k$ such that $G$ has an equitable $k$-tree-coloring is the \emph{equitable vertex arboricity} $a_{eq}(G)$ of $G$. The notion of equitable vertex arboricity was first introduced by Wu, Zhang and Li \cite{WZL}. In their paper, the authors proved that the complete bipartite graph $K_{n,n}$ has an equitable $k$-tree-coloring for every $k\geq 2\lfloor({\sqrt{8n+9}}-1)/4\rfloor$ and showed that the bound is sharp when $2n=t(t+3)$ and $t$ is odd. Note that $K_{n,n}$ admits an equitable $2$-tree-coloring. Hence a graph admitting an equitable $k$-tree-coloring may has no equitable $(k+1)$-tree-colorings. This motivates us to introduce another chromatic parameter. The \emph{strong equitable vertex arboricity} of $G$, denoted by $a^*_{eq}(G)$, is the smallest $t$ such that $G$ has an equitable $t'$-tree-coloring for every $t'\geq t$. It is easy to see that $a^*_{eq}(G)\geq a_{eq}(G)$. Concerning $a^*_{eq}(G)$, there are two interesting conjectures.

\begin{conj}\label{conj1}
$a^*_{eq}(G)\leq \lceil\frac{\Delta(G)+1}{2}\rceil$ for every graph $G$.
\end{conj}

\begin{conj}\label{conj2}
There is a constant $\zeta$ such that $a^*_{eq}(G)\leq \zeta$ for every planar graph $G$.
\end{conj}

Until now, Conjecture \ref{conj1} was confirmed for complete bipartite graphs, planar graphs with girth at least 6, planar graphs with maximum degree at least 4 and girth 5, outerplanar graphs \cite{WZL} and graphs $G$ with $\Delta(G)\geq |G|/2$ \cite{ZW}, and Conjecture \ref{conj2} was settled for planar graphs with girth at least 5 and outerplanar graphs \cite{WZL}. In particular, Wu, Zhang and Li \cite{WZL} proved that $a^*_{eq}(G)\leq 3$ for every planar graph with girth at least 5.
In this paper, we will generalize this result to Theorems \ref{main} and \ref{main-2}, and confirm Conjecture \ref{conj2} for planar graphs with all cycles of length at most 4 being independent and planar graphs without 3-cycles and adjacent 4-cycles.

\section{Main Results and their proofs}

\begin{lem}{\rm(Wu, Zhang and Li \cite{WZL})}\label{lem:label}
Let $S=\{x_1,\cdots,x_t\}$, where $x_1,\cdots,x_t$ are distinct vertices in $G$. If $G-S$ has an equitable $t$-tree-coloring and $|N(x_i)\setminus S|\leq 2i-1$ for every $1\leq i\leq t$, then $G$ has an equitable $t$-tree-coloring.
\end{lem}

\begin{lem}\label{mdg}
If $G$ is a planar graph such that all cycles of length at most $4$ are independent, then $\delta(G)\leq 3$.
\end{lem}

\begin{proof}
Suppose, to the contrary,  that $\delta(G)\geq 4$. By Euler's formula, we have $\sum_{x\in V(G)\cup F(G)}\big(d(x)-4\big)=-8$. Assign every element $x\in V(G)\cup F(G)$ an initial charge $c(x)=d(x)-4$ and define a discharging rule as follows.

\vspace{2mm}Rule. Every $5^+$-face transfer $\frac{1}{3}$ to each of its adjacent $3$-faces.

\vspace{2mm}Let $c'$ be the final charge function after discharging according to the rule. Since every $3$-face is adjacent only to $5^+$-faces by the definition of $G$, $c'(f)=3-4+3\times\frac{1}{3}=0$ for $d(f)=3$. On the other hand, every $5^+$-face $f$ is adjacent to at most $\lfloor\frac{d(f)}{2}\rfloor$ 3-faces, which implies that $c'(f)\geq d(f)-4-\frac{1}{3}\lfloor\frac{d(f)}{2}\rfloor>0$ for $d(f)\geq 5$. Therefore, $\sum_{x\in V(G)\cup F(G)}c'(x)\geq 0$, contradicting the fact that $\sum_{x\in V(G)\cup F(G)}c'(x)=\sum_{x\in V(G)\cup F(G)}c(x)=-8$.
\end{proof}

\begin{thm}\label{main}
If $G$ is a planar graph such that all cycles of length at most 4 are independent, then $a^*_{eq}(G)\leq 3$.
\end{thm}

\begin{proof}
Let $G$ be the minimal counterexample to this result and let $t\geq 3$ be an integer. To begin with, we introduce some useful structural properties of $G$.

\vspace{2mm}\noindent Proposition 1. \emph{Every $2$-vertex in $G$ is adjacent only to $7^+$-vertices.}

\proof
If there is a 2-vertex $u$ that is adjacent to a $6^-$-vertex $v$, then label $u$ and $v$ by $x_1$ and $x_t$, respectively. We now construct the set $S=\{x_1,\ldots,x_t\}$ as in Lemma \ref{lem:label} by filling the remaining unspecified positions in $S$ from highest to lowest indices properly. Actually one can easily complete it by choosing at each step a vertex of degree at most 3 in the graph obtained from $G$ by deleting the vertices already chosen for $S$. Lemma \ref{mdg} guarantees that such vertices always exist. By the minimality of $G$, $G-S$ has an equitable $t$-tree-coloring for every $t\geq 3$. Therefore, $G$ also has such a desired coloring by Lemma \ref{lem:label}. \hfill $\square$

\vspace{2mm}\noindent Proposition 2. \emph{Every $3$-vertex in $G$ is either adjacent to three $5^+$-vertices or adjacent to one $4^-$-vertex and two $7^+$-vertices.}

\proof
If there is a 3-vertex $u$ that is adjacent to a $4^-$-vertex $v$ and a $6^-$-vertex $w$, then label $u,v$ and $w$ by $x_1,x_{t-1}$ and $x_t$, respectively. By similar argument as in the proof of Proposition 1, we can construct the set $S=\{x_1,\ldots,x_t\}$ as in Lemma \ref{lem:label} and then deduce that $G$ has an equitable $t$-tree-coloring for every $t\geq 3$, a contradiction. \hfill $\square$

\vspace{2mm} Similarly, we have the following:

\vspace{2mm}\noindent Proposition 3. \emph{If there is a $3$-face $f$ that is incident with a $3$-vertex, then $f$ is either incident with two $6^+$-vertices or incident with another one $5^-$-vertex and a $8^+$-vertex.}\hfill $\square$

\vspace{2mm}\noindent Proposition 4. \emph{If there is a $4$-face $f$ that is incident with a 3-vertex, then $f$ is either incident with three $4^+$-vertices, or incident with two $5^+$-vertex, or incident with a $4$-vertex and a $7^+$-vertex.}

\proof
Let $f=u_1u_2u_3u_4$ and $d(u_1)=3$. If $f$ is not incident with three $4^+$-vertices, then there is at least one $3^-$-vertex among $u_2$, $u_3$ and $u_4$. If $\min\{d(u_2),d(u_3),d(u_4)\}=2$, then by Proposition 1, $d(u_3)=2$ and $\min\{d(u_2),d(u_4)\}\geq 7$. If $d(u_2)=3$ or $d(u_4)=3$, then by Proposition 2, $\min\{d(u_3),d(u_4)\}\geq 7$ or $\min\{d(u_2),d(u_3)\}\geq 7$, respectively. If $d(u_3)=3$, then by Proposition 2, either $\min\{d(u_2),d(u_4)\}\geq 5$ or $\min\{d(u_2),d(u_4)\}=4$ and $\min\{d(u_2),d(u_4)\}\geq 7$.
\hfill $\square$

\vspace{2mm}\noindent Proposition 5. \emph{Every $7$-vertex is adjacent to at most one $2$-vertex.}

\proof
If there is a 7-vertex $u$ that is adjacent to two 2-vertices $v$ and $w$, then label $v,w$ and $u$ by $x_1,x_{t-1}$ and $x_t$, respectively. By the similar arguments asin the proof of Proposition 1, we can
construct the set $S=\{x_1,\ldots,x_t\}$ as in Lemma \ref{lem:label}. Therefore, $G-S$ has an equitable $t$-tree-coloring by the minimality of $G$, which implies that $G$ also has such a desired coloring for every $t\geq 3$ by Lemma \ref{lem:label}. \hfill $\square$

\vspace{2mm}\noindent Proposition 6. \emph{Every $8$-vertex and every $9$-vertex is adjacent to at most four $2$-vertices.}

\proof
Let $u$ be a $k$-vertex with $8\leq k\leq 9$ and let $v_1,\ldots,v_k$ be its neighbors in $G$. Without loss of generality, assume that $v_1,v_2,v_3,v_4$ and $v_5$ are 2-vertices. Let $w_i$ be the other neighbor of $v_i$ for every $1\leq i\leq 5$.

If $t\geq 4$, then label $v_1,v_2,v_3$ and $u$ with $x_1,x_{t-2},x_{t-1}$ and $x_t$, respectively, and construct the set $S=\{x_1,\ldots,x_t\}$ as in Lemma \ref{lem:label} by the similar arguments as in the proof of Proposition 1. Therefore, $G-S$ has an equitable $t$-tree-coloring by the minimality of $G$, which implies that $G$ also has such a desired coloring for every $t\geq 4$ by Lemma \ref{lem:label}.

We now prove that $G$ has an equitable $3$-tree-coloring.
By the minimality of $G$, the graph $H=G-\{u,v_1,v_2,v_3,v_4,v_5\}$ has an equitable $3$-tree-coloring $\varphi$.
If there is one color, say $3$, that does not appear on $N(u)\setminus \{v_1,v_2,v_3,v_4,v_5\}$, then color $u$ and $v_1$ with 3, $v_2$ and $v_3$ with 1, and $v_4$ and $v_5$ with 2. One can check that the resulted coloring of $G$ is just an equitable $3$-tree-coloring.

We now assume that all of the three colors appear on $N(u)\setminus \{v_1,v_2,v_3,v_4,v_5\}$. If $d(u)=8$, then we assume that $\varphi(v_6)=1$, $\varphi(v_7)=2$ and $\varphi(v_8)=3$. If $d(u)=9$, then we assume, without loss of generality, that $\varphi(v_6)=1$, $\varphi(v_7)=2$ and $\varphi(v_8)=\varphi(v_9)=3$. The following argument is independent of the degree of $u$.
First, we color $u$ with 1. If the color on one of the vertices among $w_1,w_2,w_3,w_4$ and $w_5$, say $w_1$, is not 1, then color $v_1$ with 1, $v_2$ and $v_3$ with 2, and $v_4$ and $v_5$ with 3. If $\varphi(w_i)=1$ for every $1\leq i\leq 5$, then recolor $u$ with 2, and color $v_1$ with 2, $v_2$ and $v_3$ with 1, and $v_4$ and $v_5$ with 3. In each case, one can easily check that the resulted coloring is an equitable $3$-tree-coloring of $G$.
\hfill $\square$

\vspace{2mm}\noindent Proposition 7. \emph{Every $10$-vertex is adjacent to at most seven $2$-vertices.}

\proof
Let $u$ be a $10$-vertex and let $v_1,\ldots,v_{10}$ be its neighbors in $G$. Without loss of generality, assume that $v_1,\ldots,v_7$ and $v_8$ are 2-vertices. Let $w_i$ be the other neighbor of $v_i$ for every $1\leq i\leq 8$.
By the same argument as in the proof of Proposition 6, one can confirm that $G$ has an equitable $t$-tree-coloring for every $t\geq 4$. Thus we just need prove that $G$ admits an equitable $3$-tree-coloring.

Let $H=G-\{u,v_1,\ldots,v_8\}$. By the minimality of $G$, $H$ has an equitable 3-tree-coloring $\varphi$. Suppose that the color $3$ does not appear on $v_9$ or $v_{10}$. If there is a vertex among $w_1,\ldots,w_{8}$, say $w_1$, that is not colored by 3, then we can extend $\varphi$ to an equitable 3-tree-coloring of $G$ by coloring $u,v_1,v_2$ with 3, $v_3,v_4,v_5$ with 1, and $v_6,v_7,v_8$ with 2. If $\varphi(w_i)=3$ for every $1\leq i\leq 8$, then color $u$ with a color, say 1, that appears on $v_9$ and $v_{10}$ at most once, color $v_1$ and $v_2$ with 1, $v_3,v_4,v_5$ with 2, and $v_6,v_7,v_8$ with 3. One can easily check that the resulted coloring is an equitable $3$-tree-coloring of $G$.\hfill $\square$

\vspace{2mm} We now prove the theorem by discharging. First, assign each vertex $v$ of $G$ an initial charge $c(v)=3d(v)-10$ and each face $f$ of $G$ an initial charge $c(v)=2d(f)-10$. By Euler's formula, $\sum_{x\in V(G)\cup F(G)}c(x)=-20$. It is easy to see that there is no 1-vertices in $G$. The discharging rules we are applying are defined as follows.

\vspace{2mm} R1. Every 2-vertex receives 2 from each of its neighbors.

R2. If $u$ be a 3-vertex and $uv\in E(G)$, then $v$ sends to $u$ a charge of $\frac{1}{3}$ if $5\leq d(v)\leq 6$ and $\frac{1}{2}$ if $d(v)\geq 7$.

R3. Let $f$ be a 3-face that is incident with no 2-vertices and let $v$ be a vertex that is incident with $f$. If $4\leq d(v)\leq 7$, then $v$ sends 2 to $f$, and if $d(v)\geq 8$, then $v$ sends 4 to $f$.

R4. If $f$ is a 3-face that is incident with a 2-vertex, then $f$ receives $2$ from each of its incident $7^+$-vertices.

R5. Every 4-face receives $1$ from each of its incident $4^+$-vertices.

\vspace{2mm} Let $c'$ be the final charge after discharging. We now prove that $c'(x)\geq 0$ for every $x\in V(G)\cup F(G)$, which contradicts the fact that $\sum_{x\in V(G)\cup F(G)}c'(x)=\sum_{x\in V(G)\cup F(G)}c(x)=-20$.

If $f$ is a 3-face that is incident with a 2-vertex, then by Proposition 1, $f$ is incident with two $7^+$-vertices, which implies that $c'(v)=-4+2\times 2=0$ by R4. Suppose that $f$ is a 3-face that is incident with no 2-vertices.
If $f$ is incident with at least a $8^+$-vertex, then $c'(f)\geq -4+4=0$ by R3. If $f$ is incident only with $7^-$-vertices, then by Propositions 3, $f$ is incident with at least two $4^+$-vertices, which implies that $c'(f)\geq -4+2\times 2=0$ by R3. If $f$ is a 4-face, then by Propositions 1 and 2, $f$ is incident with at least two $4^+$-vertices, thus by R5 we have $c'(f)\geq -2+2\times 1=0$. If $f$ is a $5^+$-face, then it is easy to see that $c'(f)=c(f)\geq 0$.

If $v$ is a 2-vertex, then by Proposition 1, $v$ is adjacent to two $7^+$-vertices form which $v$ receives $2\times 2=4$ by R1, therefore $c'(v)=-4+4=0$. If $v$ is a 3-vertex, then by Proposition 2, $v$ is either
adjacent to three $5^+$-vertices which implies $c'(v)\geq -1+3\times\frac{1}{3}=0$ or adjacent to two $7^+$-vertices implying $c'(v)\geq -1+2\times\frac{1}{2}=0$ by R2. Note that every vertex in $G$ is incident with at most one $4^-$-face by the definition of $G$. If $v$ is a 4-vertex, then $c'(v)\geq 2-2=0$ by R3 and R5.
If $v$ is a 5-vertex or a 6-vertex, then by R2, R3 and R5, $c'(v)\geq 3d(v)-10-\frac{1}{3}d(v)-2>0$. If $v$ is a $7$-vertex, then $v$ is adjacent to at most one 2-vertex by Proposition 5, thus $c'(v)\geq 11-2-6\times\frac{1}{2}-2>0$ by R1--R5. If $v$ is a 8-vertex or a 9-vertex, then by Proposition 6 and R1--R5, $c'(v)\geq 3d(v)-10-4\times 2-(d(v)-4)\times \frac{1}{2}-4=\frac{1}{2}\big(5d(v)-40\big)\geq 0$.
If $v$ is a 10-vertex, then by Proposition 7 and R1--R5, $c'(v)\geq 20-7\times 2-3\times \frac{1}{2}-4>0$.

At last, we consider the vertex $v$ with $d(v)\geq 11$. If $v$ is adjacent only to 2-vertices, then $v$ is incident with no 3-faces because otherwise there would be two adjacent 2-vertices in $G$, a contradiction. Therefore, by
R1 and R5, we have $c'(v)\geq 3d(v)-10-2d(v)-1\geq 0$. If $v$ is adjacent to at most $d(v)-2$ vertices of degree 2, then by R1--R5, $c'(v)\geq 3d(v)-10-2\big(d(v)-2\big)-2\times\frac{1}{2}-4=d(v)-11\geq 0$.
Suppose that $v$ is adjacent to $d(v)-1$ vertices of degree 2. If $v$ is incident with no $4^-$-faces, then $c'(v)\geq 3d(v)-10-2\big(d(v)-1\big)-\frac{1}{2}=d(v)-\frac{17}{2}>0$ by R1 and R2. If $v$ is incident with a $4^-$-face $f$, then either $f$ is a 4-face or a 3-face that is incident with a 2-vertex. In the former case we have $c'(v)\geq 3d(v)-10-2\big(d(v)-1\big)-\frac{1}{2}-1=d(v)-\frac{19}{2}>0$ by R1, R2 and R5, and in the latter case we have $c'(v)\geq 3d(v)-10-2\big(d(v)-1\big)-\frac{1}{2}-2=d(v)-\frac{21}{2}>0$ by R1, R2 and R4.
\end{proof}

\begin{thm}\label{main-2}
If $G$ is a planar graph with girth at least 4 such that no two 4-cycles are adjacent, then $a^*_{eq}(G)\leq 3$.
\end{thm}

\begin{proof}
Let $G$ be the minimal counterexample to this result and let $t\geq 3$ be an integer. Since every planar graph with girth at least 4 contains a $3^-$-vertex, Propositions 1--7 still hold here. Therefore, the order of the following propositions we are to prove are naturally labeled from 8.

\vspace{2mm}\noindent Proposition 8. \emph{Every $11$-vertex is adjacent to at most seven $2$-vertices.}

\proof

Let $u$ be a $11$-vertex and let $v_1,\ldots,v_{11}$ be its neighbors in $G$. Without loss of generality, assume that $v_1,\ldots,v_7$ and $v_8$ are 2-vertices. Let $w_i$ be the other neighbor of $v_i$ for every $1\leq i\leq 8$.

If $t\geq 5$, then label $v_1,v_2,v_3,v_4$ and $u$ with $x_1,x_{t-3},x_{t-2},x_{t-1}$ and $x_t$, respectively, and construct the set $S=\{x_1,\ldots,x_t\}$ as in Lemma \ref{lem:label} by the similar arguments as in the proof of Proposition 1.
Therefore, $G-S$ has an equitable $t$-tree-coloring by the minimality of $G$, which implies that $G$ also has such a desired coloring for every $t\geq 5$ by Lemma \ref{lem:label}.

We now prove that $G$ has an equitable 4-tree-coloring. Let $H_1=G-\{u,v_1,\ldots,v_7\}$. By the minimality of $G$, $H_1$ has an equitable 4-tree-coloring $\varphi_1$. It is easy to see that there are at least two colors, say 1 and 2, that are used at most once on $v_8,v_9,v_{10}$ and $v_{11}$. Color $u$ with 1. If there is one vertex among $w_1,\ldots,w_7$, say $w_1$, that is not colored with 1 under $\varphi_1$, then color $v_1$ with 1, $v_2,v_3$ with 2, $v_4,v_5$ with 3, and $v_6,v_7$ with 4. If $\varphi_1(w_i)=1$ for every $1\leq i\leq 7$, then recolor $u$ with 2, color $v_1$ with 2, $v_2,v_3$ with 1, $v_4,v_5$ with 3, and $v_6,v_7$ with 4. In each case we obtain an equitable 4-tree-coloring of $G$.

At last, we show that $G$ also admits an equitable 3-tree-coloring. By the minimality of $G$, $H_2=G-\{u,v_1,\ldots,v_8\}$ has an equitable 3-tree-coloring $\varphi_2$. Without loss of generality, let 1 and 2 be the colors used at most once on $v_9,v_{10}$ and $v_{11}$. Color $u$ with 1. If there are two vertices among $w_1,\ldots,w_8$, say $w_1$ and $w_2$, that are not colored with 1 under $\varphi_2$, then color $v_1,v_2$ with 1, $v_3,v_4,v_5$ with 2, and $v_6,v_7,v_8$ with 3. On the other hand, we can assume, without loss of generality, that $\varphi_2(w_i)=1$ for every $1\leq i\leq 7$. We now recolor $u$ with 2, color $v_1,v_2$ with 2, $v_3,v_4,v_5$ with 1, and $v_6,v_7,v_8$ with 3. In each case, one can check that the resulted coloring is an equitable 3-tree-coloring of $G$. \hfill $\square$

\vspace{2mm}\noindent Proposition 9. \emph{Every $12$-vertex and every $13$-vertex is adjacent to at most ten $2$-vertices.}

\proof
Let $u$ be a $k$-vertex with $12\leq k\leq 13$ and let $v_1,\ldots,v_k$ be its neighbors in $G$. Without loss of generality, assume that $v_1,\ldots,v_{10}$ and $v_{11}$ are 2-vertices. Let $w_i$ be the other neighbor of $v_i$ for every $1\leq i\leq 11$.

By the same argument as in the proof of the above proposition, one can show that $G$ has an equitable $t$-tree-coloring for every $t\geq 5$. Let $H=G-\{u,v_1,\ldots,v_{11}\}$. By the minimality of $G$, $H$ has an equitable 4-tree-coloring $\varphi_1$ and an equitable 3-tree-coloring $\varphi_2$. It is easy to see that there is a color, say 1, that has not used on $\{w_1\}\cup N(u)\setminus \{v_{1},\ldots,v_{11}\}$ under $\varphi_1$. Hence we can extend $\varphi_1$ to an equitable 4-tree-coloring of $G$ by coloring $u,v_1,v_2$ with 1, $v_3,v_4,v_5$ with 2, $v_6,v_7,v_8$ with 3, and $v_9,v_{10},v_{11}$ with 4. On the other hand, there exists a color, say 1, that is used on $N(u)\setminus \{v_{1},\ldots,v_{11}\}$ at most once, and with which three vertices among $w_1,\ldots,w_{11}$, say $w_1,w_2$ and $w_3$, are not colored under $\varphi_2$. Therefore, $\varphi_2$ can be extended to an equitable 3-tree-coloring of $G$ by coloring $u,v_1,v_2,v_3$ with 1, $v_4,v_5,v_6,v_7$ with 2, and $v_8,v_9,v_{10},v_{11}$ with 3. Hence, $G$ admits an equitable $t$-tree-coloring for every $t\geq 3$, a contradiction. \hfill $\square$

\vspace{2mm}\noindent Proposition 10. \emph{Every $14$-vertex and every $15$-vertex is adjacent to at most thirteen $2$-vertices.}

\proof
Let $u$ be a $k$-vertex with $14\leq k\leq 15$ and let $v_1,\ldots,v_k$ be its neighbors in $G$. Without loss of generality, assume that $v_1,\ldots,v_{13}$ and $v_{14}$ are 2-vertices. Let $w_i$ be the other neighbor of $v_i$ for every $1\leq i\leq 14$.

If $t\geq 6$, then label $v_1,v_2,v_3,v_4,v_5$ and $u$ with $x_1,x_{t-4},x_{t-3},x_{t-2},x_{t-1}$ and $x_t$, respectively, and construct the set $S=\{x_1,\ldots,x_t\}$ as in Lemma \ref{lem:label} by the similar arguments as in the proof of Proposition 1. Therefore, $G-S$ has an equitable $t$-tree-coloring by the minimality of $G$, which implies that $G$ also has such a desired coloring for every $t\geq 6$ by Lemma \ref{lem:label}.

Let $H=G-\{u,v_{1},\ldots,v_{14}\}$. One can see that $H$ has an equitable 5-tree coloring $\varphi_1$ and an equitable 3-tree coloring $\varphi_2$ by the minimality of $G$.
Without loss of generality, let 1 be the color that is not used on $\{w_1\}\cup N(u)\setminus \{v_{1},\ldots,v_{14}\}$ under $\varphi_1$. We extend $\varphi_1$ to an equitable 5-tree-coloring of $G$ by coloring $u,v_1,v_2$ with 1, $v_3,v_4,v_5$ with 2, $v_6,v_7,v_8$ with 3, $v_9,v_{10},v_{11}$ with 4, and $v_{12},v_{13},v_{14}$ with 5. On the other hand,
since there is a color, say 1, that is not used on $N(u)\setminus \{v_{1},\ldots,v_{14}\}$, and with which four vertices among $w_1,\ldots,w_{14}$, say $w_1,w_2,w_3$ and $w_4$, are not colored under $\varphi_2$, we can extend $\varphi_2$ to an equitable 3-tree-coloring of $G$ by coloring $u,v_1,v_2,v_3,v_4$ with 1, $v_5,v_6,v_7,v_8,v_9$ with 2, and $v_{10},v_{11},v_{12},v_{13},v_{14}$ with 3.
Let $H'=G-\{u,v_1,\ldots,v_{11}\}$. By the minimality of $G$, $H'$ admits an equitable 4-tree-coloring $\varphi_3$. Note that there is a color, say 1, that has been used on $N(u)\setminus \{v_1,\ldots,v_{11}\}$ at most once, and with which two vertices among $w_1,\ldots,w_{11}$, say $w_1$ and $w_2$, are not colored under $\varphi_3$. Therefore, we extend $\varphi_3$ to an equitable 4-tree-coloring of $G$ by coloring $u,v_1,v_2$ with 1, $v_3,v_4,v_5$ with 2, $v_6,v_7,v_8$ with 3, and $v_{9},v_{10},v_{11}$ with 4. Hence, $G$ has an equitable $t$-tree-coloring for every $t\geq 3$, a contradiction. \hfill $\square$

\vspace{2mm} We now prove the theorem by discharging. First, assign each vertex $v$ of $G$ an initial charge $c(v)=d(v)-4$ and each face $f$ of $G$ an initial charge $c(v)=d(f)-4$. By Euler's formula, $\sum_{x\in V(G)\cup F(G)}c(x)=-8$. It is easy to see that there is no 1-vertices in $G$. The discharging rules we are applying are defined as follows.

\vspace{2mm} R1. Each $2$-vertex receives $\frac{3}{4}$ from each of its neighbors, and $\frac{1}{2}$ from each of its incident $5^+$-faces.

R2. Each $3$-vertex receives $\frac{1}{6}$ from each of its 5-neighbors or 6-neighbors, $\frac{1}{4}$ from each of its $7^+$-neighbors, and $\frac{1}{4}$ from each of it incident $5^+$-faces.

\vspace{2mm} Let $c'$ be the final charge after discharging. If $f$ is a $5^+$-face that is incident with $n$ vertices of degree 2, then $f$ is incident with at most $d(f)-2n-1$ vertices of degree 3, since 2-vertices are not adjacent to any $3^-$-vertices by Proposition 1. Hence, $c'(f)\geq d(f)-4-\frac{1}{2}n-\frac{1}{4}\big(d(f)-2n-1\big)=\frac{3}{4}\big(d(f)-5\big)\geq 0$ by R1 and R2. If $v$ is a 2-vertex, then $v$ is incident with at least one $5^+$-face by the definition of $G$, so $c'(v)\geq -2+2\times\frac{3}{4}+\frac{1}{2}=0$ by R1. If $v$ is a 3-vertex, then $v$ is incident with at least two $5^+$-faces, because otherwise there would be two adjacent 4-cycles in $G$. If $v$ is adjacent to three $5^+$-vertices, then by R2, $c'(v)\geq -1+3\times\frac{1}{6}+2\times\frac{1}{4}=0$. If $v$ is adjacent to a $4^-$-vertex, then by
Proposition 2, $v$ is adjacent to two $7^+$-vertices, which implies that $c'(v)\geq -1+2\times\frac{1}{4}+2\times\frac{1}{4}=0$ by R2. If $v$ is a 5-vertex or a 6-vertex, then $c'(v)\geq d(v)-4-\frac{1}{6}d(v)>0$ by R2, since $v$ has no 2-neighbors. If $v$ is a 7-vertex, then by Proposition 5, $v$ has at most one 2-neighbor, which implies that $c'(v)\geq 3-\frac{3}{4}-6\times\frac{1}{4}>0$ by R1 and R2. If $v$ is a 8-vertex or a 9-vertex, then by Proposition 6, R1 and R2, $c'(v)\geq d(v)-4-4\times\frac{3}{4}-\frac{1}{4}\big(d(v)-4\big)=\frac{3}{4}\big(d(v)-8\big)\geq 0$. If $v$ is a $10$-vertex, then by Proposition 7, R1 and R2, $c'(v)\geq 6-7\times\frac{3}{4}-3\times\frac{1}{4}=0$. If $v$ is a 11-vertex, then by Proposition 8, R1 and R2, $c'(v)\geq 7-7\times\frac{3}{4}-4\times\frac{1}{4}>0$.
If $v$ is a 12-vertex or a 13-vertex, then by Proposition 9, R1 and R2, $c'(v)\geq d(v)-4-10\times\frac{3}{4}-\frac{1}{4}\big(d(v)-10\big)=\frac{3}{4}\big(d(v)-12\big)\geq 0$.
If $v$ is a 14-vertex or a 15-vertex, then by Proposition 10, R1 and R2, $c'(v)\geq d(v)-4-13\times\frac{3}{4}-\frac{1}{4}\big(d(v)-13\big)=\frac{3}{4}\big(d(v)-14\big)\geq 0$.
If $v$ is a $16^+$-vertex, then $c'(v)\geq d(v)-4-\frac{3}{4}d(v)=\frac{1}{4}\big(d(v)-16\big)\geq 0$ by R1 and R2. Therefore, $\sum_{x\in V(G)\cup F(G)}c'(x)\geq 0$, a contradiction completing the proof.
\end{proof}

\end{document}